\documentclass[12pt, reqno]{amsart}
\usepackage{amsmath, amsthm, amscd, amsfonts, amssymb, graphicx, color,bm}
\usepackage[mathscr]{euscript}
\usepackage{tikz-cd}
\usepackage{hyperref}

\newtheorem{theorem}{Theorem}[section]

\newtheorem{corollary}[theorem]{Corollary}
\newtheorem{observation}[theorem]{Observation}
\theoremstyle{definition}

\theoremstyle{remark}

\newcommand{\ds}{\displaystyle}
\newcommand{\R}{{\mathbb R}}

\newcommand{\E}{{\mathbb E}}

\newcommand{\lp}{\left(}
\newcommand{\rp}{\right)}
\newcommand{\lb}{\left[}
\newcommand{\rb}{\right]}

\begin{document}

\setcounter{page}{1}

\title[Integer Programming Formulations and Probabilistic Bounds...]{Integer Programming Formulations and Probabilistic Bounds for Some Domination Parameters}

\author{Mhelmar A. Labendia$^{\rm 1,\ast}$, Clifford R. Pornia$^{\rm 2}$}

\address{$^{1,2}$ Department of Mathematics and Statistics, Mindanao State University-Iligan Institute of Technology, 9200 Iligan City, Philippines.}
\email{mhelmar.labendia@g.msuiit.edu.ph; clifford.pornia@g.msuiit.edu.ph}






\begin{abstract}
In this paper, we further study the concepts of hop domination and 2-step domination and introduce the concepts of restrained hop domination, total restrained hop domination, 2-step restrained domination, and total 2-step restrained domination in graphs. We then construct integer programming formulations and present probabilistic upper bounds for these domination parameters.
\end{abstract} 
\maketitle

\section{Introduction and Preliminaries}
Domination in graphs is one of the extensively studied concept in graph theory. This concept has historical roots as early as 1850s when European chess enthusiast studied the problem of dominating queens. The mathematical study of dominating sets has become an interest to numerous authors, in which the concept has also been used for many different applications, such as wireless network topology design \cite{Yu}, wireless sensor network \cite{Asgarnezhad}, hoc network \cite{Wu}, and many others. Different modifications related to domination in graphs have been studied by several authors, see \cite{Anusha, Mollejon, Murugan, Prabhavathy, Swaminathan, Varghese}.
The concept of total domination in graphs has also been introduced in \cite{Cockayne}. One may refer to \cite{Haynes, Haynes2} for detailed survey on domination parameters and \cite{Henning3} for thorough discussions on total domination.

Let $G = (V(G), E(G))$ be a simple undirected graph. A set $S\subseteq V(G)$ is a dominating set of $G$ if every vertex outside $S$ is adjacent to a vertex in $S$. The domination number of $G$ is the smallest cardinality of a dominating set of $G$, and is denoted by $\gamma(G)$. If every vertex in $G$ is adjacent to a vertex in $S$, then we say that $S$ is a total dominating set of $G$. Similarly, the total domination number of $G$, denoted by $\gamma_t(G)$, is the smallest cardinality of a total dominating set of $G$.

In 1999, Domke et al. \cite{Domke} initiated the study of restrained dominating set. This notion was further examined in \cite{Zelinka}. A set $S\subseteq V(G)$ is a restrained dominating set of $G$ if every vertex outside $S$ is adjacent to a vertex in $S$ as well as another vertex outside $S$. The restrained domination number of $G$ is the smallest cardinality of a restrained dominating set of $G$, and is denoted by $\gamma_r(G)$.

A new domination parameter called 2-step domination in graphs was introduced by Chartrand et al. \cite{Chartrand} and further investigated in \cite{Caro, Dror, Hersh, Zhao}. Analogous to 2-step domination, the concept of hop domination in graphs was introduced by Natarajan and Ayyaswamy \cite{Natarajan}, which was further studied by some authors, see \cite{Henning1, Henning2}. For any two vertices $u$ and $v$ of $G$, the symbol $d_G(u,v)$ or simply $d(u,v)$, is the length of a shortest path connecting vertices $u$ and $v$ in $G$, which is also referred as the distance between $u$ and $v$. The degree of a vertex $i$ in a graph $G$, denoted by $\deg_G(i)$ or $d_G(i)$, is the number of vertices adjacent to $i$ and the smallest degree among the vertices of $G$ is denoted by $\delta(G)$, or simply $\delta$. The set of vertices adjacent to $i$ in $G$ is denoted by $N(i)$ and the set of vertices at a distance 2 from $i$ in $G$ is denoted by $N_2(i)$.
 The hop degree of a vertex $i$ in a graph $G$, denoted by $\deg_h(i)$ or $d_h(i)$, is the number of vertices at distance $2$ from $i$ in $G$. The smallest hop-degree among the vertices of $G$ is denoted by $\delta_h(G)$, or simply $\delta_h$. A set $S\subseteq V(G)$ is a hop dominating set of $G$ if for every $i\in V(G)\backslash S$, there exists $j\in S$ such that $d_G(i, j)=2$. The hop domination number of $G$ is the smallest cardinality of a hop dominating set of $G$ and is denoted by $\gamma_h(G)$. A set $S\subseteq V(G)$ is a $2$-step dominating set of $G$ if for every $i\in V(G)$, there exists $j\in S$ such that $d_G(i, j)=2.$  The $2$-step domination number of $G$, denoted by $\gamma_{2\text{step}}(G)$, is the smallest cardinality of a $2$-step dominating set of $G$.

In this paper, the concept of restrained hop dominating set, total restrained hop dominating set, 2-step restrained dominating set, and
total 2-step restrained dominating set in graphs will be introduced. An integer programming (IP) formulation will be constructed for these newly
defined domination parameters and sharp upper bounds will be provided using probabilistic methods.

\section{IP Formulation}
\subsection{Hop domination and 2-step domination problems}
From now onwards, let $G = (V,E)$ be a simple undirected graph with
$V=\{1,2,\ldots,n\}$. Decision variable $x_i$ indicates whether vertex $i$ belongs to a hop dominating set $S\subseteq V$, i.e.,
\[ x_{i}= \left\{ \begin{array}{ll}
         1 & \mbox{if $i\in S$}\\
        0 & \mbox{otherwise}.
        \end{array} \right. 
\] 
Define
\[ a_{ij}= \left\{ \begin{array}{ll}
         1 & \mbox{if $i=j$ or $d(i,j)=2$}\\
        0 & \mbox{otherwise}.
        \end{array} \right. 
\] 
An IP formulation for the hop dominating set problem (HDP) can be constructed as:
\begin{align}
 \min \sum_{i=1}^n x_i &&  \label{HDP1}
 \end{align}
\qquad subject to:
\begin{align}
  \sum_{j=1}^n a_{ij}x_j\geq 1,&&  \forall i\in  V \label{HDP2}\\
  x_i\in\{0,1\},&&  \forall i\in  V \label{HDP3}
\end{align} 

\begin{theorem}
The optimal solution of the IP formulation for the HDP is equal to the hop domination number of $G$.
\end{theorem}

Next, we construct an IP formulation for the 2-step dominating set problem (2SDP). Decision
variable $x_i$ indicates whether vertex $i$ belongs to a 2-step dominating set.
Since a 2-step dominating set is a special case of a hop dominating set, we just replace constraint \eqref{HDP2} with constraint \eqref{2SDP2}. 
An IP formulation for the 2SDP can be constructed as:
\begin{align}
 \min \sum_{i=1}^n x_i &&  \label{2SDP1}
 \end{align}
\qquad subject to:
\begin{align}
  \sum_{j=1,\ j\neq i}^n a_{ij}x_j\geq 1,&&  \forall i\in  V \label{2SDP2}\\
  x_i\in\{0,1\},&&  \forall i\in  V \label{2SDP3}
\end{align}

\begin{theorem}
The optimal solution of the IP formulation for the 2SDP is equal to the 2-step domination number of $G$.
\end{theorem}

\subsection{Restrained hop and total restrained hop domination problems}

A hop dominating set $S\subseteq V(G)$ is a \emph{restrained hop dominating set} of $G$ if for every $u\in V(G)\setminus S$, there exists $v\in V(G)\setminus S$ such that $d(u,v)=1$.
For the construction of an IP formulation for restrained hop dominating set problem (RHDP), decision
variable $x_i$ indicates whether vertex $i$ belongs to a restrained hop dominating set. Define
\[ b_{ij}= \left\{ \begin{array}{ll}
         -1 & \mbox{if $i=j$}\\
        1 & \mbox{if $(i,j)\in E(G)$}\\
        0 & \mbox{otherwise}.
        \end{array} \right. 
\] 
Following the techniques employed in \cite{Duraisamy}, an IP formulation for the RHDP can be constructed as:
\begin{align}
 \min \sum_{i=1}^n x_i &&  \label{RHDP1}
 \end{align}
\qquad subject to:
\begin{align}
  \sum_{j=1}^n a_{ij}x_j\geq 1,&&  \forall i\in  V \label{RHDP2}\\
  \sum_{j=1}^n b_{ij}x_j< \deg_G(i),&&  \forall i\in  V \label{RHDP3}\\
  x_i\in\{0,1\},&&  \forall i\in  V \label{RHDP4}
\end{align}

\begin{theorem}
The optimal solution of the IP formulation for the RHDP is equal to the restrained hop domination number of $G$.
\end{theorem}
A 2-step dominating set $S\subseteq V(G)$ is a \emph{total restrained hop dominating set} of $G$ if for every $u\in V(G)\setminus S$, there exists $v\in V(G)\setminus S$ such that $d(u,v)=1$.
Since the total restrained hop dominating set is a special case for restrained
hop dominating set, we just replace constraint \eqref{RHDP2} by constraint \eqref{TRHDP2}.
Decision variable $x_i$ indicates whether vertex $i$ belongs to a total restrained hop dominating set. An IP formulation for the total restrained hop dominating set problem (TRHDP) can be constructed as:
\begin{align}
 \min \sum_{i=1}^n x_i &&  \label{TRHDP1}
 \end{align}
\qquad subject to:
\begin{align}
  \sum_{j=1, \ j\neq i}^n a_{ij}x_j\geq 1,&&  \forall i\in  V \label{TRHDP2}\\
  \sum_{j=1}^n b_{ij}x_j< \deg_G(i),&&  \forall i\in  V \label{TRHDP3}\\
  x_i\in\{0,1\},&&  \forall i\in  V \label{TRHDP4}
\end{align}

\begin{theorem}
The optimal solution of the IP formulation for the TRHDP is equal to the total restrained hop domination number of $G$.
\end{theorem}

\subsection{2-step restrained and total 2-step restrained domination problems}
Finally, we construct an IP formulation for the 2-step restrained dominating set problem (2SRDP). If $G$ is a graph, then we denote by $\text{Dist} (G;2)$ the graph consisting of the vertex set $V$ and edge set $\{uv: d_G(u,v)=2\}$.

A hop dominating set $S\subseteq V(G)$ is a \emph{2-step restrained dominating set} of $G$ if for every $u\in V(G)\setminus S$, there exists $v\in V(G)\setminus S$ such that $d(u,v)=2$. Decision variables $x_i$ indicates whether a vertex $i$ belongs to a 2-step restrained dominating set. Define
\[ c_{ij}= \left\{ \begin{array}{ll}
         -1 & \mbox{if $i=j$}\\
        1 & \mbox{if $d(i,j)=2$}\\
        0 & \mbox{otherwise}.
        \end{array} \right. 
\] 
An IP formulation for the 2SRDP can be constructed as:
\begin{align}
 \min \sum_{i=1}^n x_i &&  \label{2SRDP1}
 \end{align}
\qquad subject to:
\begin{align}
  \sum_{j=1}^n a_{ij}x_j\geq 1,&&  \forall i\in  V \label{2SRDP2}\\
  \sum_{j=1}^n c_{ij}x_j< \deg_{\text{Dist}(G;2)}(i),&&  \forall i\in  V \label{2SRDP3}\\
  x_i\in\{0,1\},&&  \forall i\in  V \label{2SRDP4}
\end{align}

\begin{theorem}
The optimal solution of the IP formulation for the 2SRDP is equal to the 2-step restrained domination number of $G$.
\end{theorem}

A 2-step dominating set $S\subseteq V(G)$ is a \emph{total 2-step restrained dominating set} of $G$ if for every $u\in V(G)\setminus S$, there exists $v\in V(G)\setminus S$ such that $d(u,v)=2$. Decision variable $x_i$ indicates whether a vertex $i$ belongs to a total 2-step restrained dominating set. Replacing constraint \eqref{2SRDP2} with constraint \eqref{T2SRDP2},
an IP formulation for the total 2-step restrained dominating set problem (T2SRDP) can be constructed as:
\begin{align}
 \min \sum_{i=1}^n x_i &&  \label{T2SRDP1}
 \end{align}
\qquad subject to:
\begin{align}
  \sum_{j=1, \ j\neq i}^n a_{ij}x_j\geq 1,&&  \forall i\in  V \label{T2SRDP2}\\
  \sum_{j=1}^n c_{ij}x_j< \deg_{\text{Dist}(G;2)}(i),&&  \forall i\in  V \label{T2SRDP3}\\
  x_i\in\{0,1\},&&  \forall i\in  V \label{T2SRDP4}
\end{align}

\begin{theorem}
The optimal solution of the IP formulation for the T2SRDP is equal to the total 2-step restrained domination number of $G$.
\end{theorem}

\section{Probabilistic Bounds}
In this section, we present probabilistic upper bounds for $\gamma_{2\text{step}}(G)$, $\gamma_{rh}(G)$, $\gamma_{trh}(G)$, $\gamma_{2sr}(G)$, and $\gamma_{t2sr}(G)$.
\subsection{2-step domination number}

Before we present the probabilistic upper bound for $\gamma_{2\text{step}}(G)$, we shall consider first the following two known results:

\begin{theorem}\cite{Henning3}\label{total domination upper bound}
If $G$ is a graph with minimum degree $\delta$, then
$$\gamma_{t}(G)\leq\dfrac{\ln\delta+1}{\delta}n.$$
\end{theorem}

\begin{observation}\cite[p.915]{Henning1}\label{hop and 2-step upper bound}
If $G$ is a graph, then the following hold:
\begin{enumerate}
\item[(i)] $\gamma_h(G)=\gamma({\rm Dist}(G;2))$.
\item[(ii)] $\gamma_{2\text{step}}(G)=\gamma_t({\rm Dist}(G;2))$.
\end{enumerate}
\end{observation}

In view of Theorem~\ref{total domination upper bound} and Observation~\ref{hop and 2-step upper bound}, we have the following result.
 
\begin{theorem}\label{2-step bound2}
If $G$ is a graph of order $n$ with $\delta_h:=\delta_h(G)\geq 1$, then
$$\gamma_{2\text{step}}(G)\leq\dfrac{\ln\delta_h+1}{\delta_h}n.$$
\end{theorem}

\begin{proof}
Note that
\begin{eqnarray*}
\gamma_{2\text{step}}(G)&=&\gamma_t(\text{Dist}(G;2))\\
&\leq& \dfrac{\ln\delta(\text{Dist}(G;2))+1}{\delta(\text{Dist}(G;2))}n\\
&=& \dfrac{\ln\delta_h+1}{\delta_h}n.
\end{eqnarray*}
This completes the proof of the theorem.
\end{proof}

\subsection{Restrained hop and total restrained hop domination numbers}\label{restrained hop section}

Given a graph $G$, a \emph{matching} $M$ in $G$ is a set of pairwise non-adjacent edges, that is, no two edges share common vertices. A \emph{maximum matching}, also known as \emph{maximum-cardinality matching}, is a matching that contains the largest possible number of edges. There may be many maximum matchings. The \emph{matching number} of a graph $G$, denoted by $\nu(G)$ is the size of a maximum matching. This number is also called the \emph{edge independence number}.

A \emph{perfect matching}, also known as \emph{complete matching}, is a matching that matches all vertices of the graph, that is, a matching is perfect if every vertex of the graph is incident to an edge of the matching.

A \emph{near-perfect matching}, or \textit{near-complete matching}, is a matching in which exactly one vertex is unmatched. It is not difficult to see that a graph can only contain a near-perfect matching when the graph has an odd number of vertices. 

The proofs of the following three results are analogous in \cite{Zverovich}.

\begin{theorem}
If $G$ is a graph of order $n$ with $\delta_h\geq 1$ and $\nu:=\nu(G)\geq \gamma_h(G)$, then
$$\gamma_{rh}(G)\leq\dfrac{2\ln(\delta_h+1)+\delta_h+3}{\delta_h+1}n-2\nu$$ and
$$\gamma_{trh}(G)\leq\dfrac{2\ln(\delta_h)+\delta_h+2}{\delta_h}n-2\nu.$$
\end{theorem}

\begin{proof} Let $D$ be a minimum hop dominating set of $G$. Then $|D|=\gamma_h(G)=:\gamma_h$. 
Let $M$ be a maximum matching of $G$. Then $\nu:=\nu(G)=|M|$. Let $M=\{e_1,e_2,\ldots,e_\nu\}$. If all of the end vertices of the edges in $M$ are not in $D$, we may replace an edge of $M$ with an edge such that one of the end vertex is in $D$. Hence, we may assume that the first $k$ edges of $M$ have at least one end vertex contained in $D$. It follows that $\nu-k$ edges in $M$ have both end vertices contained in $V(G)\setminus D$. Since $k\leq |D|=\gamma_h(G)$, $\nu-\gamma_h\leq \nu-k$. Since $\nu\geq \gamma_h$, at least $\nu-\gamma_h$ edges in $M$ have both end vertices contained in $V(G)\setminus D$. Let $D'=\{u\in V(G)\setminus D: \mbox{$u$ is not an end vertex of the $\nu-k$ edges in $M$}\}$. Let $S:=D\cup D'$. Then $S$ is a restrained hop dominating set of $G$. Now,
\begin{eqnarray*}
\gamma_{rh}(G)&\leq&|S|\\
&=&|D\cup D'|\\
&=&n-2\lp \nu-k\rp\\
&\leq&n-2\lp\nu-\gamma_h\rp\\
&=&n+2\gamma_h-2\nu.
\end{eqnarray*}
In view of \cite[Theorem~12, p.926]{Henning1}, $\gamma_h\leq\dfrac{\ln(\delta_h+1)+1}{\delta_h+1} n$. Hence, 
$$\gamma_{rh}(G)\leq\dfrac{2\ln(\delta_h+1)+\delta_h+3}{\delta_h+1}n-2\nu.$$
This proves the first part of the theorem.

Next, let $D^{(t)}$ be a minimum 2-step dominating set of $G$. Then $|D^{(t)}|=\gamma_{2\text{step}}(G):=\gamma_{2\text{step}}$. Using the same technique employed above, we can construct a total restrained hop dominating set $S_t:=D^{(t)}\cup D'$ of $G$ so that
\begin{eqnarray*}
\gamma_{trh}(G)&\leq&n+2\gamma_{2\text{step}}-2\nu.
\end{eqnarray*}
By Theorem~\ref{2-step bound2}, $\gamma_{2\text{step}}\leq\dfrac{\ln \delta_h+1}{\delta_h}n$. Hence, 
$$\gamma_{trh}(G)\leq\dfrac{2\ln \delta_h+\delta_h+2}{\delta_h}n-2\nu.$$
This completes the proof of the theorem.
\end{proof}

\begin{corollary}\label{matching restrained hop}
If a graph $G$ of order $n$ has a perfect matching with $\nu(G)\geq \gamma_h(G)$, then
$$\gamma_{rh}(G)\leq\dfrac{\ln(\delta_h+1)+1}{\delta_h+1}2n$$ and
$$\gamma_{trh}(G)\leq\dfrac{\ln\delta_h+1}{\delta_h}2n.$$
\end{corollary}

\begin{corollary}
If a graph $G$ of order $n$ has a near-perfect matching with $\nu(G)\geq \gamma_h(G)$, then
$$\gamma_{rh}(G)\leq\dfrac{\ln(\delta_h+1)+1}{\delta_h+1}2n+1$$ and
$$\gamma_{trh}(G)\leq\dfrac{\ln\delta_h+1}{\delta_h}2n+1$$
\end{corollary}

\bigskip
Next, we improve the assumption in Corollary~\ref{matching restrained hop}.  From now onwards, let $\R$ be the set of real numbers and let $C^n = \left\{\textbf{\textit{p}}=(p_1,\ldots,p_n):p_i\in\R, 0\leq p_i\leq 1\right\}$. 

Let $f_{rh}:C^n\to\R$ be the function defined by
\begin{eqnarray*}
f_{rh}(\textbf{\textit{p}})&=&\sum_{i=1}^n p_i +\sum_{i=1}^n (1-p_i)\ds\prod_{j\in N_2(i)}(1-p_j)\\
&&\quad+\sum_{i=1}^n (1-p_i)\lp 1-(1-p_i)\ds\prod_{j\in N_2(i)}(1-p_j)\rp\\
&&\quad\quad\times\ds\prod_{j\in N(i)}\lb p_j+(1-p_j)\ds\prod_{k\in N_2(j)}(1-p_k)\rb.
\end{eqnarray*}

\begin{theorem}\label{restrained hop bound}
If $G$ is a graph of order $n$, then 
$$\gamma_{rh}(G)=\ds\min_{\textbf{\textit{p}}\in C^n} f_{rh}(\textbf{\textit{p}})$$
\end{theorem}

\begin{proof}
Let $G$ be a graph with vertex set $V = \{1,2,\ldots,n\}$. We pick randomly and independently each vertex $i\in V$ with probability $p_i$, where $0\leq p_i\leq1$, to form a set $X\subseteq V$, that is, ${\mathbb P}(i\in X):={\mathbb P}(\{i\in V: i\in X\})=p_i$.
Let $Z=\{i\notin X: N_2(i)\cap X=\varnothing\}$ and $Y=\{i\notin X\cup Z: N(i)\subseteq X\cup Z\}$. Consider $D=X\cup Z\cup  Y$. 

First, we show that $D$ is a restrained hop dominating set of $G$.
Let $u\in V\setminus D$. Then $u\notin X\cup Z\cup Y$ so that $N_2(u)\cap X\neq\varnothing$. This means that there exists $v\in X\subset D$ such that $d_G(u,v)=2$. Also, $N(u)\nsubseteq X\cup Z$. It follows that there exists $s\in N(u)$ such that $s\notin X\cup Z$. If $s\in Y$, then $N(s)\subseteq X\cup Z$, which implies that $u\in N(s)\subseteq X\cup Z$, a contradiction. Hence, $s\notin Y$. Thus, $s\in V\setminus D$. Accordingly, $D$ is a restrained hop dominating set of $G$.

Define $X_i:=X(i)$ by
\[ X_i= \left\{ \begin{array}{ll}
         1 & \mbox{if $i\in X$}\\
        0 & \mbox{otherwise}.
        \end{array} \right. 
\] 
Similarly, define $Z_i:=Z(i)$ (resp., $Y_i:=Y(i)$) by
\[ \mbox{$Z_i$ (resp., $Y_i$)}= \left\{ \begin{array}{ll}
         1 & \mbox{if $i\in Z$ (resp., $i\in Y$)}\\
        0 & \mbox{otherwise}.
        \end{array} \right. 
\] 
Then $|X|=\ds\sum_{i=1}^n X_i$, $|Z|=\ds\sum_{i=1}^n Z_i$, and $|Y|=\ds\sum_{i=1}^n Y_i$. Note that
$$\E[X_i]={\mathbb P}(i\in X)=p_i,$$
$$\E[Z_i]={\mathbb P}(i\in Z)=(1-p_i)\ds\prod_{j\in N_2(i)}(1-p_j),$$ and
\begin{eqnarray*}
\E[Y_i]&=&{\mathbb P}(i\in Y)\\
&=&(1-p_i)\lp 1-(1-p_i)\ds\prod_{j\in N_2(i)}(1-p_j)\rp\ds\prod_{j\in N(i)}\lb p_j+(1-p_j)\ds\prod_{k\in N_2(j)}(1-p_k)\rb.
\end{eqnarray*}
Hence,
\begin{eqnarray*}
\E\lb |D|\rb&=& \E\lb |X|\rb+\E\lb |Z|\rb++\E\lb |Y|\rb\\
&=&\sum_{i=1}^n p_i +\sum_{i=1}^n (1-p_i)\ds\prod_{j\in N_2(i)}(1-p_j)\\
&&\quad+\sum_{i=1}^n (1-p_i)\lp 1-(1-p_i)\ds\prod_{j\in N_2(i)}(1-p_j)\rp\\
&&\quad\quad\times\ds\prod_{j\in N(i)}\lb p_j+(1-p_j)\ds\prod_{k\in N_2(j)}(1-p_k)\rb\\
&=&f_{rh}(p_1,p_2,\ldots,p_n).
\end{eqnarray*}
This means that there exists a restrained hop dominating set of $G$ of cardinality at most $\E[|D|]$. Thus, $\gamma_{rh}(G)\leq\ds\min_{\textbf{\textit{p}}\in C^n} f_{rh}(\textbf{\textit{p}})$.

Next, let $D'$ be a minimum restrained hop dominating set of $G$. Then $|D'|=\gamma_{rh}(G)$. Let $\textbf{\textit{p}}'=\lp p_1',p_2',\ldots,p_n'\rp$, where
\[ p_i'= \left\{ \begin{array}{ll}
         1 & \mbox{if $i\in D'$}\\
        0 & \mbox{otherwise}.
        \end{array} \right. 
\]  
Then $$f_{rh}(\textbf{\textit{p}}')=\sum_{i=1}^n p_i'=|D'|=\gamma_{rh}(G).$$
Accordingly, $\gamma_{rh}(G)=\ds\min_{\textbf{\textit{p}}\in C^n} f_{rh}(\textbf{\textit{p}})$.
\end{proof}

\begin{theorem}\label{restrained hop bound2}
If $G$ is a graph of order $n$ with $\delta_h\geq 1$, then
$$\gamma_{rh}(G)\leq\dfrac{\ln(\delta_h+1)+1}{\delta_h+1}2n.$$
\end{theorem}

\begin{proof}
Let $\textbf{\textit{p}}=(p,p,\ldots,p)\in C^n$. For each $i\in\{1,2,\ldots,n\}$, $\delta_h\leq d_h(i)$ and $1-x\leq e^{-x}$ for $x\in\R$, and we have
\begin{eqnarray*}
f_{rh}(\textbf{\textit{p}})&=&\sum_{i=1}^n p+\sum_{i=1}^n(1-p)^{d_h(i)+1}\\
&&\qquad+\sum_{i=1}^n (1-p)\lp 1-(1-p)^{d_h(i)+1}\rp\lp p+(1-p)^{d_h(i)+1}\rp^{\deg(i)}\\
&\leq& np+n(1-p)^{\delta_h+1}+n(1-p)\lp p+(1-p)^{\delta_h+1}\rp^{\delta}\\
&\leq& np+n(1-p)^{\delta_h+1}+n(1-p)\lp p+(1-p)^{\delta_h+1}\rp\\
&\leq& np+n(1-p)^{\delta_h+1}+np(1-p)+n(1-p)^{\delta_h+2}\\
&\leq& 2np + 2n(1-p)^{\delta_h+1}\\
&\leq& 2np+2ne^{-p(\delta_h+1)}.
\end{eqnarray*}
Note that the function $g(p)=2np+2ne^{-p(\delta_h+1)}$, $p\in[0,1]$, has an absolute minimum value at $p=\dfrac{\ln (\delta_h+1)}{\delta_h+1}$. Observe that $0<\dfrac{\ln(\delta_h+1)}{\delta_h+1}<1$. Let $\textbf{\textit{p}}'=(p',p',\ldots,p')\in C^n$, where $p'=\dfrac{\ln(\delta_h+1)}{\delta_h+1}$. Then by Theorem~\ref{restrained hop bound},
$$\gamma_{rh}(G)\leq f(\textbf{\textit{p}}')\leq 2np'+2ne^{-p'(\delta_h+1)}=\dfrac{\ln(\delta_h+1)+1}{\delta_h+1}2n.$$
This completes the proof of the theorem.
\end{proof}

Next, we improve the upper bound in Theorem~\ref{restrained hop bound2} but we need to put additional assumption. This technique is also used in \cite{Zverovich}.

\begin{theorem}
If $G$ be a graph of order $n$ with $\delta,\delta_h\geq 1$ and $n<\ds\frac{\delta\delta_h}{\ln\delta_h+1}$, then
$$\gamma_{rh}(G)\leq\dfrac{\ln(\delta_h+1)+1}{\delta_h+1}n$$ and
$$\gamma_{trh}(G)\leq\dfrac{\ln\delta_h+1}{\delta_h}n$$
\end{theorem}

\begin{proof} In view of \cite[Theorem~12, p.926]{Henning1}, let $D$ be a hop dominating set of $G$ with
$$|D|\leq \dfrac{\ln(\delta_h+1)+1}{\delta_h+1}n.$$
Since $n<\ds\frac{\delta\delta_h}{\ln\delta_h+1}$, $\delta>\ds\frac{\ln\delta_h+1}{\delta_h}n$.
Let $u\in V(G)\setminus D$. Then
$$\deg_G(u)\geq\delta>\frac{\ln\delta_h+1}{\delta_h}n>\frac{\ln(\delta_h+1)+1}{\delta_h+1}n=|D|.$$
This means that there exists $v\in V(G)\setminus D$ such that $uv\in E(G)$. Thus, $D$ is a restrained hop dominating set of $G$. Accordingly, $$\gamma_{rh}(G)\leq\dfrac{\ln(\delta_h+1)+1}{\delta_h+1}n.$$
This proves the first part of the theorem.

Next, in view of Theorem~\ref{2-step bound2}, let $D^{(t)}$ be a 2-step dominating set of $G$ with
$$|D^{(t)}|\leq \dfrac{\ln\delta_h+1}{\delta_h}n.$$
Let $u\in V(G)\setminus D$. Then
$$\deg_G(u)\geq\delta>\frac{\ln\delta_h+1}{\delta_h}n=|D^{(t)}|.$$
This means that there exists $v\in V(G)\setminus D$ such that $uv\in E(G)$. Thus, $D^{(t)}$ is a total restrained hop dominating set of $G$. Accordingly, $$\gamma_{trh}(G)\leq\dfrac{\ln\delta_h+1}{\delta_h}n.$$
This completes the proof.
\end{proof}

\subsection{2-step restrained and total 2-step restrained domination numbers}

Given a graph $G$, a \emph{hop matching} $H$ in $G$ is a set of paths of size two such that no two paths share a common end vertex. A \emph{maximum hop matching} is a hop matching that contains the largest possible number of paths of size two. There may be many maximum hop matchings. The \emph{hop matching number} of a graph $G$, denoted by $\nu_h(G)$, is the cardinality of a maximum hop matching.

A \emph{perfect hop matching} or \emph{complete hop matching}, is a hop matching such that every vertex of the graph is an end vertex of an element of the hop matching.

A \emph{near-perfect hop matching}, or \textit{near-complete hop matching}, is a hop matching such that exactly one vertex of the graph is not an end vertex of an element of the hop matching.

\begin{theorem}
If $G$ is a graph of order $n$ with $\delta_h\geq 1$ and ${\rm Dist}(G;2)$ has no isolated vertex, then
$$\gamma_{2sr}(G)\leq\dfrac{2\ln(\delta_h+1)+\delta_h+3}{\delta_h+1}n-2\nu_h(G)$$ and
$$\gamma_{t2sr}(G)\leq\dfrac{2\ln\delta_h+\delta_h+2}{\delta_h}n-2\nu_h(G)$$
\end{theorem}

\begin{proof} Let $D$ be a minimum hop dominating set of $G$. Then $|D|=\gamma_h(G)=:\gamma_h$. 
Let $H$ be a maximum hop matching of $G$. Then $\nu_h:=\nu_h(G)=|H|$. Let $H=\{P_2^1,P_2^2,\ldots,P_2^{\nu_h}\}$. If all of the end vertices of the elements in $H$ are not in $D$, we may replace an element of $H$ with a path of size two such that one of the end vertex is in $D$. Hence, we may assume that the first $k$ paths in $H$ have at least one end vertex contained in $D$. It follows that $\nu_h-k$ paths in $H$ have both end vertices contained in $V(G)\setminus D$. Since $k\leq |D|=\gamma_h(G)$, $\nu_h-\gamma_h\leq \nu_h-k$. Since ${\rm Dist}(G;2)$ has no isolated vertex, by \cite[Theorem~2, p.2314]{Henning2}, $\nu_h\geq \gamma_h$ so that at least $\nu_h-\gamma_h$ paths in $H$ have both end vertices contained in $V(G)\setminus D$. Let $D'=\{u\in V(G)\setminus D: \mbox{$u$ is not an end vertex of the $\nu_h-k$ paths in $H$}\}$. Let $S:=D\cup D'$. Then $S$ is a 2-step restrained dominating set of $G$. Now,
\begin{eqnarray*}
\gamma_{2sr}(G)&\leq&|S|\\
&=&|D\cup D'|\\
&=&n-2\lb \nu_h-k\rb\\
&\leq&n-2\lb\nu_h-\gamma_h\rb\\
&=&n+2\gamma_h-2\nu_h.
\end{eqnarray*}
In view of \cite[Theorem~12, p.926]{Henning1}, $\gamma_h\leq\dfrac{\ln(\delta_h+1)+1}{\delta_h+1}n$. Hence, 
$$\gamma_{2sr}(G)\leq\dfrac{2\ln(\delta_h+1)+\delta_h+3}{\delta_h+1}n-2\nu_h.$$
This proves the first part of the theorem.

Next, let $D^{(t)}$ be a minimum 2-step dominating set of $G$. Then $|D^{(t)}|=\gamma_{2\text{step}}(G):=\gamma_{2\text{step}}$. Using the same technique employed above, we can construct a total 2-step restrained dominating set $S_t:=D^{(t)}\cup D'$ of $G$ so that
\begin{eqnarray*}
\gamma_{t2sr}(G)&\leq&n+2\gamma_{2\text{step}}-2\nu_h.
\end{eqnarray*}
By Theorem~\ref{2-step bound2}, $\gamma_{2\text{step}}\leq\dfrac{\ln \delta_h+1}{\delta_h}n$. Hence, 
$$\gamma_{t2sr}(G)\leq\dfrac{2\ln \delta_h+\delta_h+2}{\delta_h}n-2\nu_h.$$
\end{proof}

\begin{corollary}\label{hop matching restrained hop}
If a graph $G$ of order $n$ has a perfect hop matching, then
$$\gamma_{2sr}(G)\leq\dfrac{\ln(\delta_h+1)+1}{\delta_h+1}2n$$ and
$$\gamma_{t2sr}(G)\leq\dfrac{\ln\delta_h+1}{\delta_h}2n.$$
\end{corollary}

\begin{corollary}
If a graph $G$ of order $n$ has a near-perfect hop matching, then
$$\gamma_{2sr}(G)\leq\dfrac{\ln(\delta_h+1)+1}{\delta_h+1}2n+1$$ and
$$\gamma_{t2sr}(G)\leq\dfrac{\ln\delta_h+1}{\delta_h}2n+1$$
\end{corollary}

Next, we improve the assumption in Corollary~\ref{hop matching restrained hop}.
Let $f_{2sr}:C^n\to\R$ be the function defined by
\begin{eqnarray*}
f_{2sr}(\textbf{\textit{p}})&=&\sum_{i=1}^n p_i +\sum_{i=1}^n (1-p_i)\ds\prod_{j\in N_2(i)}(1-p_j)\\
&&\quad+\sum_{i=1}^n (1-p_i)\lp 1-(1-p_i)\ds\prod_{j\in N_2(i)}(1-p_j)\rp\\
&&\quad\quad\times\ds\prod_{j\in N_2(i)}\lb p_j+(1-p_j)\ds\prod_{k\in N_2(j)}(1-p_k)\rb.
\end{eqnarray*}
The proof of the following result is similar with Theorem~\ref{restrained hop bound} with $Y=\{i\notin X\cup Z: N(i)\subseteq X\cup Z\}$ replaced with $Y=\{i\notin X\cup Z: N_2(i)\subseteq X\cup Z\}$.

\begin{theorem}\label{2-step restrained bound}
If $G$ is a graph of order $n$, then 
$$\gamma_{2sr}(G)=\ds\min_{\textbf{\textit{p}}\in C^n} f_{2sr}(\textbf{\textit{p}})$$
\end{theorem}

\begin{theorem}\label{2-step restrained bound2}
If $G$ is a graph of order $n$ with $\delta_h\geq 1$, then
$$\gamma_{2sr}(G)\leq\dfrac{\ln(\delta_h+1)+1}{\delta_h+1}2n.$$
\end{theorem}

Similar with section~\ref{restrained hop section}, the upper bound in Theorem~\ref{restrained hop bound2} can be improved but we need to put additional assumption.

\begin{theorem}
If $G$ be a graph of order $n$ with $\delta_h\geq 1$ and $n<\ds\frac{\delta_h^2}{\ln\delta_h+1}$, then
$$\gamma_{2sr}(G)\leq\dfrac{\ln(\delta_h+1)+1}{\delta_h+1}n$$ and
$$\gamma_{t2sr}(G)\leq\dfrac{\ln\delta_h+1}{\delta_h}n.$$
\end{theorem}

\section{Conclusion and Recommendations}
This paper has introduced the concepts of restrained hop, total restrined hop, 2-step restrained, and total 2-step restrained dominating sets and constructed their corresponding IP formulations. Sharp upper bounds has also been provided using probabilistic methods. A worthwhile direction for further study is to consider the complexity of these domination parameters.


\end{document}